\documentclass{amsart}
\usepackage{graphicx}
\usepackage{caption}
\usepackage{amsfonts}
\newtheorem{tm}{Theorem}

\newtheorem{defi}{Definition}

\newtheorem{rem}{Remark}

\newtheorem{lm}{Lemma}

\newtheorem{nota}{Notation}

\begin{document}

\title{A nonrealization theorem in the context of Descartes' rule of signs}
\author{Hassen Cheriha, Yousra Gati and Vladimir Petrov Kostov}
\address{Universit\'e C\^ote d'Azur, LJAD, France and 
University of Carthage, EPT - LIM, Tunisia}
\email{hassan.cheriha@unice.fr}
\address{University of Carthage, EPT - LIM, Tunisia}
\email{yousra.gati@gmail.com} 
\address{Universit\'e C\^ote d'Azur, LJAD, France} 
\email{vladimir.kostov@unice.fr}
\begin{abstract}
For a real degree $d$ polynomial $P$ with all 
nonvanishing coefficients, with $c$ sign changes and $p$ sign preservations 
in the sequence of its coefficients ($c+p=d$), Descartes' rule of signs says that $P$ has $pos\leq c$ positive and 
$neg\leq p$ negative roots, where $pos\equiv c($\, mod $2)$ and 
$neg\equiv p($\, mod $2)$. For $1\leq d\leq 3$, 
for every possible choice of the 
sequence of signs of coefficients of $P$ (called sign pattern) 
and for every pair 
$(pos, neg)$ satisfying these conditions there exists a polynomial $P$ 
with exactly 
$pos$ positive and $neg$ negative roots (all of them simple); that is, all these cases are realizable. This is not true for 
$d\geq 4$, yet for $4\leq d\leq 8$ (for these degrees the exhaustive answer to the question of realizability is known)  in all nonrealizable cases either $pos=0$ or $neg=0$. It was conjectured that this 
is the case for any $d\geq 4$. For $d=9$, we show a counterexample to this conjecture: for the sign pattern $(+,-,-,-,-,+,+,+,+,-)$ and the pair $(1,6)$ there exists no polynomial with $1$ positive, $6$ negative simple roots and a complex conjugate pair and, up to equivalence, this is the only case for $d=9$. 

\end{abstract}
\maketitle 

\section{Introduction}

In his work La G\'eom\'etrie published in 1637, Ren\'e Descartes (1596-1650) announces his classical rule of signs which says that for the real polynomial $P(x,a):=x^d+a_{d-1}x^{d-1}+\cdots +a_0$, the number $c$ of sign changes in the sequence of its coefficients serves as an upper bound for the number of its positive roots. When roots are counted 
with multiplicity, then the number of positive roots has the same parity as $c$. One can apply these results to the polynomial $P(-x)$ to obtain an upper bound on the number of negative roots of $P$. For a given $c$, one can find  polynomials $P$ with $c$ sign changes  with exactly $c$, $c-2$, $c-4$, $\ldots$ positive roots. One should observe that by doing so one does not impose any restrictions on the number of negative roots. 

\begin{rem}
{\rm It is mentioned in \cite{AlFu} that 18th century authors used to count roots with multiplicity while omitting 
the parity conclusion; later this conclusion was attributed (see \cite{Ca}) 
to a paper of Gauss 
of 1828 (see \cite{Ga}), although it is absent there, but  
was published by Fourier in 1820 (see p. 294 in \cite{Fo}).}
\end{rem}

In the present paper we consider polynomials $P$ without zero coefficients. We denote by $p$ the number of sign preservations in the sequence of coefficients of $P$, and by $pos_P$ (resp. $neg_P$) the number of 
positive and negative roots of $P$. Thus the following condition must be fulfilled: 

\begin{equation}\label{posneg}
pos_P\leq c~~~,~~~pos_P\equiv c\, (\, {\rm mod~}2)~~~,~~~neg_P\leq p~~~,~~~
neg_P\equiv p\, (\, {\rm mod~}2)~. 
\end{equation}

\begin{defi}
{\rm A {\em sign pattern} is a finite sequence $\sigma$ of $(\pm )$-signs; we assume that the leading sign of $\sigma$ is $+$. For a given sign pattern of length $d+1$ with $c$ sign changes and $p$ sign preservations, 
we call $(c,p)$ its {\em Descartes pair}, $c+p=d$. For a given sign pattern $\sigma$ with Descartes pair $(c,p)$, we call $(pos, neg)$ an {\em admissible pair} for $\sigma$ if conditions (\ref{posneg}), with $pos_P=pos$ and $neg_P=neg$, are satisfied.}
\end{defi}

It is natural to ask the following question: {\em Given a sign pattern $\sigma$ 
of length $d+1$ and an admissible pair $(pos, neg)$ can one find a degree $d$ real monic 
polynomial the signs of whose coefficients define the sign pattern $\sigma$ 
and which has exactly $pos$ simple positive and exactly $neg$ simple negative 
roots ?} When the answer to the question is positive we say that the couple $(\sigma ,(pos, neg))$ is {\em realizable}. 

For $d=1$, $2$ and $3$, the answer to this question is positive, but for $d=4$ D.~J.~Grabiner showed that this is not the case, see \cite{Gr}. Namely, for the sign pattern $\sigma ^*:=(+,+,-,+,+)$ (with Descartes pair $(2,2)$), 
the pair $(2,0)$ is admissible, see (\ref{posneg}), but the couple $(\sigma ^*,(2,0))$ is not realizable. Indeed, for a monic polynomial $P_4:=x^4+a_3x^3+\cdots +a_0$ with signs of the coefficients 
defined by $\sigma ^*$ and having exactly two positive roots $u<v$ one has $a_j>0$ for 
$j\neq 2$, $a_2<0$ and $P_4((u+v)/2)<0$. Hence $P_4(-(u+v)/2)<0$ because $a_j((u+v)/2)^j=a_j(-(u+v)/2)^j$, $j=0$, $2$, $4$ and $0<a_j((u+v)/2)^j=-a_j(-(u+v)/2)^j$, $j=1$,~$3$. As $P_4(0)=a_0>0$, 
there are two negative roots $\xi <-(u+v)/2<\eta$ as well. 

\begin{defi}\label{z2act}
{\rm We define {\em the standard $\mathbb{Z}_2\times \mathbb{Z}_2$-action} on couples of the form (sign pattern, 
admissible pair) by its two generators. Denote by $\sigma (j)$ the $j$th component of the sign pattern $\sigma$. The 
first of the generators replaces the sign pattern $\sigma$ by $\sigma ^r$, where $\sigma ^r$ stands for the {\em reverted} (i.e. read from the back) sign pattern multiplied by $\sigma (1)$, and keeps the same pair $(pos, neg)$. This generator corresponds to the fact that the polynomials $P(x)$ and $x^dP(1/x)/P(0)$ are both monic and have the same numbers of positive and negative roots. The second generator exchanges $pos$ with $neg$ and changes the signs of $\sigma$ corresponding to the monomials of odd (resp. even) powers if $d$ is even (resp. odd); the rest of the signs are preserved. We denote the new sign pattern by $\sigma _m$. This generator corresponds to the fact 
that the roots of the polynomials (both monic) $P(x)$ and $(-1)^dP(-x)$ are mutually opposite, and if $\sigma$ is the sign pattern of $P$, then $\sigma _m$ is the one of $(-1)^dP(-x)$. }
\end{defi}

\begin{rem}
{\rm For a given sign pattern $\sigma$ and an admissible pair $(pos, neg)$, the couples $(\sigma , (pos, neg))$, $(\sigma ^r, (pos, neg))$, $(\sigma _m, (neg, pos))$ 
and $((\sigma _m)^r, (neg, pos))$ are simultaneously realizable or not. One has $(\sigma _m)^r=(\sigma ^r)_m$.}
\end{rem}
Modulo the standard $\mathbb{Z}_2\times \mathbb{Z}_2$-action Grabiner's example is the only nonrealisable couple (sign pattern, admissible pair) for $d=4$. All cases of couples (sign pattern, admissible pair) for $d=5$ and $6$ 
which are not realizable are described in \cite{AlFu}. For $d=7$, this is done in \cite{FoKoSh} and for $d=8$ in \cite{FoKoSh} and \cite{Ko2}. For $d=5$, there is a single nonrealizable case (up to the $\mathbb{Z}_2\times \mathbb{Z}_2$-action). The sign pattern is $(+,+,-,+,-,-,)$ and the admissible pair is $(3,0)$. For $n=6$, $7$ and $8$ there are respectively $4$, $6$, and $19$ nonrealizable cases. In all of them one of the numbers $pos$ or $neg$ is $0$. 
In the present paper we show that for $d=9$ this is not so.  

\begin{nota}\label{sigmam}
{\rm For $d=9$, we denote by $\sigma ^0$ the following sign pattern 
(we give on the first and
third lines below respectively the sign
patterns $\sigma ^0$ and $\sigma ^0_m$ while the line in the middle
indicates the positions of the monomials
of odd powers):}

$$\begin{array}{cccccccccccccccc}
\sigma ^0&=&(&+&-&-&-&-&+&+&+&+&-&)\\ &&&9&&7&&5&&3&&1&&&\\ 
\sigma ^0_m&=&(&+&+&-&+&-&-&+&-&+&+&)\end{array}$$
{\rm In a sense $\sigma ^0$ is centre-antisymmetric -- it consists of 
one plus, five minuses, five pluses and one minus.}
\end{nota}

\begin{tm}\label{maintm}
(1) The sign pattern $\sigma ^0$ is not realizable with the admissible pair~$(1,6)$.

(2) Modulo the standard $\mathbb{Z}_2\times \mathbb{Z}_2$-action, for $d \leq 9$, this is the only nonrealizable couple (sign pattern, admissible pair) in which both components of the admissible pair are nonzero. 
\end{tm}

\begin{rem}
{\rm It is shown in \cite{Ko11} that for $d=11$, the admissible pair $(1,8)$ is not realizable with the sign pattern (+ - - - - - + + + + + -). Hence Theorem~\ref{maintm} shows an example of a nonrealisable couple, with both components of the admissible pair different from zero, in the least possible degree (namely, 9).}
\end{rem}

Section~\ref{seccomments} contains comments concerning the above result and realizability of sign patterns and admissible pairs in general. Section~\ref{prelim} contains some technical lemmas which allow to simplify the 
proof of Theorem~\ref{maintm}. The plan of the proof of part (1) of Theorem~\ref{maintm} is explained in Section~\ref{method}. The proof results from several lemmas whose proofs can be found in Section~\ref{prlm}. The proof of part (2) of Theorem~\ref{maintm} is given in Section~\ref{secpr2}. 


\section{Comments}\protect\label{seccomments}

It seems that the problem to classify, for any degree $d$, all couples (sign pattern, admissible pair) which are not realizable, is quite difficult. This is confirmed by Theorem~\ref{maintm}. For the moment, only certain sufficient conditions for realizability or nonrealizability have been formulated: 
\begin{itemize}
\item in \cite{FoKoSh} and \cite{KoSh} series of nonrealizable cases were found, for $d \geq 4$, even and for $d \geq 5$, odd respectively;
\item in \cite{FoKoSh} sufficient conditions are given for the nonrealizability of sign patterns with exactly two sign changes.
\item in \cite{YoVlHa} sufficient conditions are given for realizability the nonrealizability of sign patterns with exactly two sign changes.
\end{itemize}
\begin{rem}\label{rem8g1}
{\rm For $d \leq 8$, all couples (sign pattern, admissible pairs) with $pos \geq 1$, $neg \geq 1$, are realizable. That is, in the examples of nonrealizability given in \cite{FoKoSh} and \cite{KoSh} one has either $pos=0$ or $neg=0$, so the question to construct an example of nonrealizability with $pos \neq 0 \neq neg$ was a challenging one}
\end{rem}

The result in \cite{FoKoSh} about sign patterns with exactly two sign changes, consisting of $m$ pluses followed by $n$ minuses followed by $q$ pluses, with $m+n+q=d+1$, is formulated in terms of the following quantity:  

$$\kappa :=\frac{d-m-1}{m}\cdot \frac{d-q-1}{q}~.$$

\begin{lm}\label{lm2changes}
For $\kappa \geq 4$, such a sign pattern is not realizable with the admissible 
pair $(0,d-2)$. The sign pattern is realizable with any admissible pair of 
the form $(2,v)$.
\end{lm}
Lemma~\ref{lm2changes} coincides with Proposition~6 of \cite{FoKoSh}. One can construct new realizable cases with the help of the following concatenation lemma (see its proof in \cite{FoKoSh}):

\begin{lm}\label{lmconcat}
Suppose that the
monic polynomials $P_j$ of degrees $d_j$ and with sign 
patterns of the form $(+,\sigma _j)$, $j=1$, $2$ (where $\sigma _j$ contains 
the last $d_j$ components of the corresponding sign pattern) realize the 
pairs $(pos_j, neg_j)$. Then 

(1) if the last position of $\sigma _1$ is $+$, then for any $\varepsilon >0$ 
small enough, the polynomial $\varepsilon ^{d_2}P_1(x)P_2(x/\varepsilon )$ 
realizes the sign pattern $(+,\sigma _1,\sigma _2)$ and the pair 
$(pos_1+pos_2,neg_1+neg_2)$;

(2) if the last position of $\sigma _1$ is $-$, then for any $\varepsilon >0$ 
small enough, the polynomial $\varepsilon ^{d_2}P_1(x)P_2(x/\varepsilon )$ 
realizes the sign pattern $(+,\sigma _1,-\sigma _2)$ and the pair 
$(pos_1+pos_2,neg_1+neg_2)$ (here $-\sigma _2$ is obtained from $\sigma _2$ 
by changing each $+$ by $-$ and vice versa).
\end{lm} 

\begin{rem}
{\rm If Lemma~\ref{lmconcat} were applicable to the case treated in Theorem~\ref{maintm}, then this case would be realizable and Theorem~\ref{maintm} would be false. We show here that Lemma~\ref{lmconcat} is indeed inapplicable. It suffices to check the cases $\deg P_1\geq 5$, $\deg P_2\leq 4$ due to the centre-antisymmetry of $\sigma ^0$ and the possibility to use the $\mathbb{Z}_2\times \mathbb{Z}_2$-action. In all these cases the sign pattern of the polynomial $P_1$ has exactly two sign changes (including the first sign $+$, the four minuses that follow and the next between 
one and four pluses). With the notation from Lemma~\ref{lm2changes}, these cases are $m=1$, $n=4$, $q=1$, $\ldots$, $4$. The respective values of $\kappa$ are $9$, $6$, $5$ and $9/2$. All of them are $> 4$. By Descartes' rule the polynomial $P_1$ can have either $0$ or $2$ positive roots. In the case of $2$ positive roots, Lemma~\ref{lmconcat} implies that its concatenation with $P_2$ has at least $2$ positive roots which is a contradiction. Hence $P_1$ has no positive roots. The polynomials $P_1$ and $P_2$ define sign patterns with $3+q-1$ and $4-q$ sign preservations respectively. The polynomial $P_1$ has $\leq 1+(q-1)$ negative roots (see Lemma~\ref{lm2changes}) and $P_2$ has $\leq 4-q$ ones. Therefore he concatenation of $P_1$ and $P_2$ has $\leq 6$ negative roots and a polynomial realizing the couple $(\sigma ^0,(1,6))$ (if any) could not be represented as a concatenation of $P_1$ and $P_2$. This, of course, does not a priori mean that such a polynomial does not exist.}
\end{rem}

\section{Preliminaries\protect\label{prelim}}

\begin{nota} 
{\rm By $S$ we denote the set of tuples $a \in \mathbb{R}^{9}$ for which the polynomial $P(x,a)=x^{9}+a_{8}x^{8}+\cdots +a_0$ realizes the pair $(1,6)$ and the signs of its coefficients define the sign pattern $\sigma_0$. 

We denote by $T$ the subset of $S$ for which $a_{8}=-1$. For a 
polynomial $P \in S$, the conditions $a_{9}=1$, $a_{8}=-1$ can be obtained by 
rescaling the variable $x$ and by multiplying $P$ by a nonzero constant ($a_{9}$ is the leading coefficient of $P$).}
\end{nota}

\begin{lm}\label{lmnot0}
For $a\in \bar{S}$, one has $a_j\neq 0$ for $j=7,6,3,2$, and one 
does not have
$a_4=0$ and $a_5=0$ simultaneously.
\end{lm}
\begin{proof}[Proof of Lemma~\ref{lmnot0}]
For $a_j=0$ (where $j$ is one of the
indices $7,6,3,2$) there are less than $6$ sign changes in the 
sign pattern
$\sigma ^0_m$. Descartes' rule of signs implies that the polynomial $P(.,a)$ 
has less than
$6$ negative roots counted with multiplicity. The same is true for $a_5=a_4=0$. 
\end{proof}

\begin{lm}\label{lmnot0bis}
For $a\in \bar{S}$, one has $a_0\neq 0$.
\end{lm}

\begin{rem}\label{rem1or3}
{\rm A priori the set $\bar{S}$ can contain polynomials with all roots real and 
nonzero. The positive ones can be either a triple root or a double 
and a simple roots (but not three simple roots). If $a\in S$, then $P(x,a)$ has the maximal possible number of negative roots (equal to the number of sign preservations in the sign pattern). If $a^{'}\in \bar{S}$, then the polynomial $Q(x,a^{'})$ is the limit of polynomials $Q(x,a)$ with $a\in S$. In the limit as a $\rightarrow a^{'}$, the complex conjugate pair can become a double positive, but not a double negative root, because there are no 8 sign preservations in the sign pattern.}
\end{rem}

\begin{proof}[Proof of Lemma~\ref{lmnot0bis}]
Suppose that for $P \in \bar{S}$, one has $a_0=0$ and for $j \neq 0$, $a_j \neq 0$. Hence the polynomial $P_1:= P/x$ has 6 negative roots and either 0 or 2 positive roots. We show that 0 positive roots is impossible. Indeed, the polynomial  $P_1$ defines a sign pattern with exactly 2 sign changes. Suppose that all negative roots are distinct. If $P_1$ has no positive roots, then one can apply Lemma~\ref{lm2changes}, according to which, as one has $\kappa =  9/2 > 4$, such a polynomial does not exist. If $P_1$ has a negative root $-b$ of multiplicity $m>1$, then its perturbation 

$$P_{1,\epsilon}:=(x+b+ \epsilon) P_1/(x+b)~,~0 < \epsilon \ll 1~,$$ 
defines the same sign pattern and instead of the root $-b$ of multiplicity $m$ has a root $-b$ of multiplicity $m-1$ and a simple root $b- \epsilon$. After finitely many such perturbations,  one is in the case when all negative roots are distinct. 

If $P_1$ has 2 positive roots, then this is a double positive root $g$, see Remark~\ref{rem1or3}. 
In this case, we add to $P_1$ a linear term $\pm \epsilon x $ (with $\epsilon$ small enough in order not to change the sign pattern) to make the double root bifurcate into a complex conjugate pair. The sign is chosen depending on whether $P_1$ has a minimum or a maximum at $g$.  
After this, if there are multiple negative roots, we apply perturbations of the form $P_{1,\epsilon}$. 

Suppose that $a_1=a_0=0$, and that for $j \geq 2$, $a_j \neq 0$. Then one considers the polynomial $P_2:=P/x^2$. It defines a sign pattern with two sign changes and one has $\kappa = 5 > 4$. Hence it has 2 positive roots, otherwise one obtains a contradiction with Lemma~\ref{lm2changes}.  

Suppose now that exactly one of the coefficients $a_4$ or $a_5$ is 0. We assume this to be $a_4$, for $a_5$ the reasoning is similar. Suppose also that either $a_1 \neq 0$, $a_0 = 0$ or $a_1 = a_0 = 0$, and that for $j \geq 2$, $j \neq 4$, one has $a_j \neq 0$. We treat in detail the case $a_1 \neq 0$, $a_0 = 0$, the case $a_1 = a_0 = 0$ is treated by analogy. We first make the double positive root if any bifurcate into a complex conjugate pair as above. This does not change the coefficient $a_4$. After this instead of perturbations $P_{1,\epsilon}$ we use perturbations preserving the condition $a_4 = 0$.
Suppose that $P_1=(x-b)^mQ_1Q_2$, where $Q_1$ and $Q_2$ are monic polynomials, deg $Q_2 = 2$, $Q_2$ having a complex conjugate pair of roots, $Q_1$ having $6-m$ negative roots counted with multiplicity. Then we set: 

$$ P_1 \mapsto P_1+ \epsilon (x-b)^{m-1} (x+h_1)(x+h_2)Q_1~, $$ 
where the real numbers $h_i$ are distinct, different from any of the roots of $P$ and chosen in such a way that the coefficient $\delta$ of $x^3$ of $P_1$ is 0.
Such a choice is possible, because all coefficients of the polynomial $(x+b)^{m-1}Q_1$ are positive, hence $\delta$ is of the form $A+(h_1+h_2)B+Ch_1h_2$, where $A>0$, $B>0$ and $C>0$. The result of the perturbation is a polynomial $P_1$ having six negative distinct roots and a complex conjugate pair; its coefficient of $x^3$ is 0. By adding a small positive number to this coefficient, one obtains a polynomial $P_1$ with roots as before and defining the sign pattern $(+ - - - - + + + +)$. For this polynomial one has $\kappa = 9/2 > 4$ which contradicts Lemma~\ref{lm2changes}.  

In the case $a_1=a_0=0$, the polynomial $P_1$ thus obtained has five negative distinct roots, a complex conjugate pair and a root at 0. One adds small positive numbers to its constant term and to its coefficient of $x^3$ and one proves in the same way that such a polynomial does not exist. 
\end{proof}

\begin{rem}\label{rem2lem}
{\rm One deduces from Lemmas~\ref{lmnot0} and \ref{lmnot0bis} that for a polynomial in $\bar{T}$ exactly one of the following conditions holds true: 

(1) all its coefficients are nonvanishing;

(2) exactly one of them is vanishing and this coefficient is either $a_1$ or $a_4$ or $a_5$; 

(3) exactly two of them are vanishing, and these are either $a_1$ and $a_4$ or $a_1$ and $a_5$.

 }
\end{rem}

\begin{lm}\label{lmnotexist}
There exists no real degree 9 polynomial satisfying the following conditions: 
\begin{itemize}
\item the signs of its coefficients define the sign pattern $\sigma^0$,
\item it has a complex conjugate pair with nonpositive real part,
\item it has a single positive root,
\item it has negative roots of total multiplicity 6.
\end{itemize}
\end{lm}

\begin{proof}
Suppose that such a monic polynomial $P$ exists. We can write $P$ in the form 
$P=P_1P_2P_3$, where $\deg P_1=6$. 

All roots of $P_1$ are negative hence 
$P_1=\sum _{j=0}^6\alpha _jx^j$, $\alpha _j>0$, $\alpha _6=1$; $P_2=x-w$, $w>0$; 
$P_3=x^2+\beta _1x+\beta _0$, $\beta _j\geq 0$, $\beta _1^2-4\beta _0<0$. 

By Descartes' rule of signs, the polynomial $P_1P_2=\sum _{j=0}^7\gamma _jx^j$, 
$\gamma _7=1$,  
has exactly one sign 
change in the sequence of its coefficients. It is clear that 
as $0>a_{8}=\gamma _6+\beta _1$, and as $\beta _1\geq 0$, one must have 
$\gamma _6<0$. But then $\gamma _j<0$ for $j=0$, $\ldots$, $6$. For 
$j=2$, 3 and $4$, one has 
$a_j=\gamma _{j-2}+\beta _1\gamma _{j-1}+\beta _0\gamma _j<0$ which means that 
the signs of $a_j$ do not define the sign pattern $\sigma ^0$.
\end{proof}   

\begin{rem}\label{remnotexist}
{\rm It follows from Lemma~\ref{lmnotexist} that polynomials of $\bar{T}$ can only 
have negative roots of total multiplicity $6$ and positive roots 
of total multiplicity $1$ or $3$ (i.e., either one simple, 
or one simple and one double 
or one triple positive root); these polynomials have no root at $0$ (Lemma~\ref{lmnot0bis}). 
Indeed, when approaching the boundary of $T$, 
the complex conjugate pair can coalesce to form a double positive 
(but never nonpositive) root; the latter might eventually coincide with the 
simple positive root.}
\end{rem}

\section{Plan of the proof of part (1) of Theorem~\protect\ref{maintm}\protect\label{method}}

Suppose that there exists a monic polynomial $P(x,a^*)|_{a^*_{8}=-1}$ with signs of its coefficients 
defined by the sign pattern $\sigma ^0$, with $6$ distinct negative, a simple positive and two complex conjugate roots. 

Then for $a$ close to $a^*\in \mathbb{R}^{8}$, all polynomials $P(x,a)$ share with $P(x,a^*)$ these properties. Therefore the interior of the set $T$ is nonempty. 
In what follows we denote by $\Gamma$ the connected component of 
$T$ to which $a^*$ belongs. 
Denote by $-\delta$ the value of $a_7$ for $a=a^*$ (recall that this value 
is negative). 

\begin{lm}\label{lmKd} 
There exists a compact set $K\subset \bar{\Gamma}$ containing all points of 
$\bar{\Gamma}$ with $a_7\in [-\delta ,0)$. Hence there exists $\delta _0>0$ 
such that for every point of $\bar{\Gamma}$, one has 
$a_7\leq -\delta _0$, and for at least one point of $K$ and for no point 
of $\bar{\Gamma}\backslash K$, the equality $a_7=-\delta _0$ holds.
\end{lm}

\begin{proof}
Suppose that there exists an unbounded sequence $\{ a^n\}$ of values 
$a\in \bar{\Gamma}$ with $a_7^n\in [-\delta ,0)$. Hence one can perform 
rescalings $x\mapsto \beta _nx$, $\beta _n>0$, such that the largest of the 
moduli of the coefficients of the monic polynomials $Q_n:=(\beta _n)^{-9}P(\beta _nx,a^n)$ equals $1$. These polynomials belong to $\bar{S}$, not necessarily to $\bar{T}$ because $a_{8}$ after the rescalings, in general, 
is not equal to $-1$. The coefficient of $x^7$ in $Q_n$ equals $a_7^n/(\beta _n)^2$. The sequence $\{ a^n\}$ si unbounded, so there exists a subsequence $\beta _{n_k}$ tending to $\infty$. 
This means that the sequence of monic polynomials $Q_{n_k}\in \bar{S}$ with 
bounded coefficients has a polynomial 
in $\bar{S}$ with $a_7=0$ as one of its limit points 
which contradicts Lemma~\ref{lmnot0}. 

Hence the moduli of the roots and the tuple of coefficients $a_j$ of $P(x,a)\in \bar{\Gamma}$ 
with $a_7\in [-\delta ,0)$ remain bounded from which the existence of $K$ and $\delta _0$ follows. 
\end{proof}

The above lemma implies the existence of a polynomial $P_0\in \bar{\Gamma}$ 
with $a_7=-\delta _0$. We say that $P_0$ is $a_7$-{\em maximal}. Our aim is 
to show that no polynomial of $\bar{\Gamma}$ is $a_7$-maximal which 
contradiction will be the proof of Theorem~\ref{maintm}.

\begin{defi}
{\rm A real univariate polynomial is {\em hyperbolic} if it has only real 
(not necessarily simple) roots. We denote by $H\subset \bar{\Gamma}$ the set 
of hyperbolic polynomials in $\bar{\Gamma}$. Hence these are monic degree 
$9$ polynomials having positive and negative roots of respective total 
multiplicities $3$ and $6$ (vanishing roots are impossible 
by Lemma~\ref{lmnot0}). 
By $U\subset \bar{\Gamma}$ we denote the set of 
polynomials in $\bar{\Gamma}$ having a complex conjugate pair, 
a simple positive 
root and negative roots of total multiplicity $6$. 
Thus $\bar{\Gamma}=H\cup U$ and $H\cap U=\emptyset$. We denote by 
$U_0$, $U_2$, $U_{2,2}$, $U_3$ and $U_4$ the subsets of $U$ for which the 
polynomial $P\in U$ has respectively $6$ simple negative roots, one double 
and $4$ simple negative roots, 
at least two negative roots of multiplicity $\geq 2$, one triple and 
$3$ simple negative roots 
and a negative root of multiplicity $\geq 4$.}
\end{defi}

The following lemma on hyperbolic polynomials is proved in \cite{Ko11}. It is used in the proofs of the other lemmas.

\begin{lm}\label{lmhyp}
Suppose that $V$ is a  
hyperbolic polynomial of degree $d\geq 2$ with no root at $0$. Then: 

(1) $V$ does not have two or more 
consecutive vanishing coefficients. 

(2) If $V$ has a vanishing coefficient, then the signs of 
its surrounding two coefficients are opposite.

(3) The number of  
positive (of negative) roots of $V$ is equal to the number of sign changes 
in the sequence of its coefficients (in the one of $V(-x)$). 
\end{lm}

By a sequence of lemmas we consecutively decrease the set of possible $a_7$-{\em maximal} polynomials until in the end it turns out that this set must be empty. 
The proofs of the lemmas of this section except Lemma~\ref{lmKd} are given in Sections~\ref{prlm} (Lemmas~\ref{lm24} -- \ref{lmH3bis}), \ref{prlmbis} (Lemma~\ref{lmH4bis}) and \ref{prlmter} (Lemmas~\ref{lmH4ter} -- \ref{lmH6}).

\begin{lm}\label{lm24}
(1) No polynomial of $U_{2,2}\cup U_4$ is $a_7$-maximal.

(2) For each polynomial of $U_3$, there exists a polynomial of $U_0$ 
with the same values of $a_7$, $a_5$, $a_4$ and $a_1$.

(3) For each polynomial of $U_0\cup U_2$, there exists a polynomial of 
$H\cup U_{2,2}$ with the same values of $a_7$, $a_5$, $a_4$ and $a_1$.
\end{lm}

Lemma~\ref{lm24} implies that if there exists an $a_7$-maximal polynomial in 
$\bar{\Gamma}$, 
then there exists such a polynomial in $H$. So from now on, we aim 
at proving that $H$ contains no such polynomial hence $H$ and $\bar{\Gamma}$ 
are empty.

\begin{lm}\label{lmH2}
There exists no polynomial in $H$ having exactly two distinct real roots.
\end{lm}

\begin{lm}\label{lmHtriple}
The set $H$ contains no polynomial having one triple positive root and  
negative roots of total multiplicity $6$.
\end{lm}

Lemma~\ref{lmHtriple} and Remark~\ref{rem1or3} 
imply that a polynomial in $H$ (if any) 
satisfies the following condition: 
\vspace{2mm}

{\em Condition A.} Any polynomial $P \in H$ has a double and a simple positive roots 
and negative roots of total multiplicity~$6$.

\begin{lm}\label{lmH3}
There exists no polynomial $P\in H$ having exactly three distinct real roots 
and satisfying the conditions $\{ a_1=0, a_4=0\}$ or 
$\{ a_1=0, a_5=0\}$.
\end{lm}

It follows from the lemma and from Lemma~\ref{lmnot0} that a polynomial  
$P\in H$ having exactly three distinct real roots (hence a double and a 
simple positive and an $6$-fold negative one) 
can satisfy at most one of the 
conditions $a_1=0$, $a_4=0$ and 
$a_5=0$.

\begin{lm}\label{lmH3bis}
No polynomial in $H$ having exactly three distinct real roots is $a_7$-maximal.
\end{lm}

Thus an $a_7$-maximal polynomial in $H$ (if any) must satisfy 
Condition A and have at least four 
distinct real roots.

\begin{lm}\label{lmH4bis}
The set $H$ contains no polynomial having a double and a simple positive roots 
and exactly two distinct negative roots of total multiplicity $6$, and which 
satisfies either the conditions $\{ a_1=a_4=0\}$ or $\{ a_1=a_5=0\}$.
\end{lm}

At this point we know that an $a_7$-maximal polynomial of $H$ satisfies 
Condition A and one of the two following conditions:
\vspace{2mm}

{\em Condition B.} It has exactly four distinct real roots and satisfies 
exactly one or none of the equalities 
$a_1=0$, $a_4=0$ or $a_5=0$.
\vspace{2mm}

{\em Condition C.} It has at least five distinct real roots.

\begin{lm}\label{lmH4ter}
The set $H$ contains no $a_7$-maximal 
polynomial satisfying Conditions A and B.
\end{lm}

Therefore an $a_7$-maximal polynomial in $H$ (if any) must satisfy 
Conditions A and~C.

\begin{lm}\label{lmH5}
The set $H$ contains no $a_7$-maximal 
polynomial having exactly five distinct real roots.
\end{lm}

\begin{lm}\label{lmH6}
The set $H$ contains no $a_7$-maximal 
polynomial having at least six distinct real roots.
\end{lm}

Hence the set $H$ contains no $a_7$-maximal polynomial at all. It follows from 
Lemma~\ref{lm24} that there is no such polynomial in $\bar{\Gamma}$. 
Hence~$\bar{\Gamma}=\emptyset$.

\section{Proofs of Lemmas~\protect\ref{lmhyp}, 
\protect\ref{lm24}, 
\protect\ref{lmH2}, \protect\ref{lmHtriple}, 
\protect\ref{lmH3} and \protect\ref{lmH3bis}
\protect\label{prlm}}

\begin{proof}[Proof of Lemma~\ref{lmhyp}:]
Part (1). Suppose that a hyperbolic polynomial $V$ 
with two or more vanishing coefficients 
exists. If $V$ is degree $d$ hyperbolic, then $V^{(k)}$ is also hyperbolic 
for $1\leq k<d$. Therefore we can assume that $V$ is of the form 
$x^{\ell}L+c$, where $\deg L=d-\ell$, $\ell \geq 3$, 
$L(0)\neq 0$ and $c=V(0)\neq 0$. 
If $V$ is hyperbolic and 
$V(0)\neq 0$, 
then such is also $W:=x^dV(1/x)=cx^d+x^{d-\ell}L(1/x)$ 
and also $W^{(d-\ell )}$ which 
is of the form $ax^{\ell}+b$, $a\neq 0\neq b$. However given that $\ell \geq 3$, 
this polynomial is not hyperbolic. 

For the proof of part (2) we use exactly the same reasoning, but with 
$\ell =2$. The polynomial $ax^2+b$, $a\neq 0\neq b$, is hyperbolic if and 
only if $ab<0$.

To prove part (3) we consider the sequence of coefficients of  
$V:=\sum _{j=0}^dv_jx^j$, $v_0\neq 0\neq v_d$. 
Set $\Phi :=\sharp \{ k|v_k\neq 0\neq v_{k-1},v_kv_{k-1}<0\}$, 
$\Psi :=\sharp \{ k|v_k\neq 0\neq v_{k-1},v_kv_{k-1}>0\}$ and 
$\Lambda :=\sharp \{ k|v_k=0\}$. Then $\Phi +\Psi +2\Lambda =d$. By Descartes' 
rule of signs the number of 
positive (of negative) roots of $V$ is $pos_V\leq \Phi +\Lambda$ 
(resp. $neg_V\leq \Psi +\Lambda$). As $pos_V+neg_V=d$, one must have 
$pos_V=\Phi +\Lambda$ and $neg_V=\Psi +\Lambda$. It remains to notice that 
$\Phi +\Lambda$ is the number of sign changes in the sequence of coefficients 
of $V$ (and $\Psi +\Lambda$ of $V(-x)$), see part (2) of the lemma.

\end{proof}

\begin{proof}[Proof of Lemma~\ref{lm24}:] Part (1). 
A polynomial of $U_{2,2}$ or 
$U_4$ respectively is 
representable in the form:

$$P^{\dagger}:=(x+u)^2(x+v)^2S\Delta ~~~\, \, {\rm and}~~~\, \, 
P^*:=(x+u)^4S\Delta ~,$$
where $\Delta :=(x^2-\xi x+\eta )(x-w)$ and $S:=x^2+Ax+B$. 
All coefficients $u$, $v$, $w$, $\xi$, $\eta$, $A$, $B$ 
are positive and $\xi ^2-4\eta <0$ (see Lemma~\ref{lmnotexist}); for 
$A$ and $B$ this follows from the fact that 
all roots of $P^{\dagger}/\Delta$ 
and $P^*/\Delta$ are negative. (The roots of $x^2+Ax+B$ 
are not necessarily different from $-u$ and $-v$.) We consider 
the two Jacobian matrices 

$$J_1:=(\partial (a_{8},a_7,a_1,a_4)/\partial (\xi ,\eta ,w,u))~~~{\rm and}~~~
J_2:=(\partial (a_{8},a_7,a_1,a_5)/\partial (\xi ,\eta ,w,u))~.$$
In the case of $P^{\dagger}$ their determinants equal 

$$\begin{array}{ccll}
\det J_1&=& (A^2u^2v+2A^2uv^2+2Au^2v^2+Auv^3+2ABu^2+5ABuv&\\ &&
+2ABv^2+3Bu^2v+2Buv^2+Bv^3+2B^2u+B^2v) \Pi~,\\ \\
\det J_2&=&(A^2uv+Au^2v+2Auv^2+2ABu&\\ &&
+ABv+2Bu^2+4Buv+2Bv^2) \Pi~,\end{array}$$ 
where $\Pi :=-2v(w+u)(-\eta -w^2+w\xi )(\xi u+\eta +u^2)$. 

These determinants are nonzero. Indeed, each of the factors is either a sum of 
positive terms or equals 
$-\eta -w^2+w\xi <-\xi ^2/4-w^2+w\xi =-(\xi /2-w)^2\leq 0$. Thus one can choose 
values of $(\xi ,\eta ,w,v)$ close to the initial one ($u$, $A$ and $B$ remain fixed) to obtain any values of $(a_{8},a_7,a_1,a_4)$ or 
$(a_{8},a_7,a_1,a_5)$ close to the initial one. In particular, with $a_{8}=-1$, 
$a_1=a_4=0$ or $a_{8}=-1$, $a_1=a_5=0$ while $a_7$ can have values larger than 
the initial one. Hence this is not an $a_7$-maximal polynomial. (If 
the change of the value of $(\xi ,\eta ,w,v)$ is small enough, the values 
of the coefficients $a_j$, $j=0$, $2$, $3$, $5$ or $4$ and $6$ 
can change, but their signs remain the same.) The same 
reasoning is valid for $P^*$ as well in which case one has 
 
$$\begin{array}{ccll}
\det J_1&=&(3A^2u^2+3Au^3+9ABu+6Bu^2+3B^2)M~,\\ \\
\det J_2&=&(A^2u+3Au^2+3AB+8Bu)M~,\end{array}$$ 
with $M:=-4u^2(w+u)(-\eta -w^2+w\xi )(\xi u+\eta +u^2)$. 

To prove part (2), we observe that if the triple root of $P\in U_3$ 
is at $-u<0$, then in case when $P$ is increasing (resp. decreasing) 
in a neighbourhood of $-u$ 
the polynomial $P-\varepsilon x^2(x+u)$ (resp. $P+\varepsilon x^2(x+u)$), 
where $\varepsilon >0$ 
is small enough, has three simple roots close to $-u$; it belongs to 
$\bar{\Gamma}$, its coefficients $a_j$, $2\neq j\neq 3$, are the same as 
the ones of $P$, the signs of $a_2$ and $a_3$ are also the same. 

For the proof of part (3), we observe first that 1) for $x<0$ 
the polynomial $P$ has three maxima and three minima and 
2) for $x>0$ one of the following 
three things holds true: either $P'>0$, or there is a double 
positive root $\gamma$ of 
$P'$, or $P'$ has two positive roots $\gamma _1<\gamma _2$ 
(they are both either smaller than  
or greater than the positive root of $P$). Suppose first that $P\in U_0$. 
Consider the family of polynomials $P-t$, $t\geq 0$. Denote by $t_0$ 
the smallest value of $t$ for which one of the three things happens: either  
$P-t$ has a double negative root $v$ (hence a local maximum), or 
$P-t$ has a triple 
positive root $\gamma$ or $P-t$ has a double and a simple positive roots 
(the double one is at $\gamma _1$ or $\gamma _2$). 
In the second and third cases 
one has $P-t_0\in H$. In the first case, if $P-t_0$ has another double 
negative root, then $P-t_0\in U_{2,2}$ and we are done. 
If not, then consider the family of 
polynomials 

$$P_s:=P-t_0-s(x^2-v^2)^2(x^2+v^2)=P-t_0-s(x^6-v^2x^4-x^2v^4+v^6)~~~,~~~s\geq 0~.$$
The polynomial $-(x^6-v^2x^4-x^2v^4+v^6)$ has double real roots at $\pm v$ 
and a complex conjugate pair. It has the same signs of the 
coefficients of 
$x^6$, $x^4$ and $1$ as $P-t_0$ and $P$. The rest of the coefficients of 
$P-t_0$ and $P_s$ are the same. As $s$ increases, the value of $P_s$ 
for every $x\neq \pm v$ decreases. So for some $s=s_0>0$ for the first time 
one has either $P_s\in U_{2,2}$ (another local maximum of $P_s$ 
becomes a double negative root) or $P_s\in H$ ($P_s$ has positive roots 
of total multiplicity $3$, but not three simple ones). This proves part (3) for 
$P\in U_0$.

If $P\in U_2$ and the double negative root is a local minimum, then 
the proof of part (3) is just the same. If this is a local maximum, 
then one skips the construction of the family $P-t$ and starts constructing 
the family $P_s$ directly. 
\end{proof}

\begin{proof}[Proof of Lemma~\ref{lmH2}:]
Suppose that such a polynomial exists. Then it must be of the form 
$P:=(x+u)^6(x-w)^3$, $u>0$, $w>0$. The conditions $a_{8}=-1$ and  
$a_1>0$ read:

$$6u-3w=-1~~~\, \, {\rm and}~~~\, \, 3u^5w^2(u-2w)>0~.$$
In the plane of the variables $(u,w)$ the domain $\{ u>0, w>0, u-2w>0\}$ 
does not intersect the line $6u-3w=-1$ which proves the lemma.
\end{proof}

\begin{proof}[Proof of Lemma~\ref{lmHtriple}:]

Represent the polynomial in the form $P=(x+u_1)\cdots (x+u_6)(x-\xi )^3$, where 
$u_j>0$ and $\xi >0$. The numbers $u_j$ are not necessarily distinct. 
The coefficient $a_{8}$ then equals $u_1+\cdots +u_6-3\xi$. The condition 
$a_{8}=-1$ implies $\xi =\xi _*:=(u_1+\cdots +u_6+1)/3$. Denote by 
$\tilde{a}_1$ the coefficient $a_1$ 
expressed as a function of $(u_1,\ldots ,u_6,\xi )$. Using 
computer algebra (say, MAPLE) one finds $27\tilde{a}_1|_{\xi =\xi _*}$:

$$27\tilde{a}_1|_{\xi =\xi _*}=-(-3u_1\cdots u_6+X+Y)(u_1+\cdots +u_6+1)^2~,$$
where $Y:=u_1\cdots u_6(1/u_1+\cdots +1/u_6)$ and 
$X:=u_1\cdots u_6\sum _{1\leq i,j\leq 6,i\neq j}u_i/u_j$ (the sum $X$ contains $30$ 
terms). We show that $a_1<0$ which by contradiction proves the lemma. 
The factor 
$(u_1+\cdots +u_6+1)^2$ is positive. The factor $\Xi :=-3u_1\cdots u_6+X+Y$ 
contains a single monomial with a negative coefficient, namely, 
$-3u_1\cdots u_6$. Consider the sum 

$$\begin{array}{clcc}
&-3u_1\cdots u_6+u_1^2u_3u_4u_5u_6+u_2^2u_3u_4u_5u_6+u_1u_3^2u_4u_5u_6+u_2u_3u_4^2u_3u_5u_6&&\\ 
+&u_1u_3u_4u_5^2u_6+u_2u_3u_4u_5u_6^2&&\\ \\
=&u_3u_4u_5u_6((u_1-u_2)^2+u_1u_2)+u_3u_4u_5u_6((u_1-u_2)^2+u_1u_2)&&\\
+&u_3u_4u_5u_6((u_1-u_2)^2+u_1u_2)&>&0\end{array}$$
(the second and third monomials are in $X$). Hence $\Xi$ is representable 
as a sum of positive quantities, so $\Xi >0$ and $a_1<0$. 
\end{proof}

\begin{proof}[Proof of Lemma~\ref{lmH3}:]
Suppose that such a polynomial exists. Then it must be of the  
form $(x+u)^6(x-w)^2(x-\xi )$, 
where $u>0$, $w>0$, $\xi >0$, $w\neq \xi$. 
One checks numerically (say, using MAPLE), 
for each of the two systems of algebraic 
equations 
$a_{8}=-1$, $a_1=0$, $a_4=0$ and $a_{8}=-1$, $a_1=0$, $a_5=0$, that each 
real solution $(u, w, \xi )$ or $(u,v,w)$ contains a nonpositive component. 
\end{proof}

\begin{proof}[Proof of Lemma~\ref{lmH3bis}:] 
Making use of Condition A formulated after Lemma~\ref{lmHtriple}, 
we consider only polynomials of the form 
$(x+u)^6(x-w)^2(x-\xi )$. 
Consider the Jacobian matrix 

$$J_1^*:=(\partial (a_{8},a_7,a_1)/\partial (u,w,\xi ))~.$$ 
Its determinant equals 
$-12u^4(u+w)(u-5w)(\xi -w)(k+u)$. All factors except $u-5w$ are nonzero. 
Thus for $u\neq 5w$, one has $\det J_1\neq 0$, so one 
can fix the values of $a_{8}$ and $a_1$ 
and vary the one of $a_7$ arbitrarily close to the initial one by choosing 
suitable values of $u$, $w$ and $\xi$. Hence the polynomial is not 
$a_7$-maximal. For $u=5w$, one has $a_3=-2500w^5(\xi + 5w )<0$ which is 
impossible. Hence 
there exist no $a_7$-maximal polynomials 
which satisfy only the condition $a_1=0$ or none of the conditions $a_1=0$, 
$a_4=0$ or $a_5=0$. To see that there exist no such polynomials satisfying only 
the condition $a_4=0$ or $a_5=0$ one can consider the matrices 
$J_4^*:=(\partial (a_{8},a_7,a_4)/\partial (u,w,\xi ))$ and 
$J_5^*:=(\partial (a_{8},a_7,a_5)/\partial (u,w,\xi ))$. Their determinants 
equal respectively 

$$-60u(u+w)(2u-w)(\xi -w)(\xi +u)~~~\, {\rm and}~~~\, 
-12u(u+w)(5u-w)(\xi -w)(\xi +u)~.$$
They are nonzero respectively for $2u\neq w$ and $5u\neq w$, in which cases 
in the same way we conclude that the polynomial is not $w_7$-maximal. 
If $u=w/2$, then $a_1=-(1/64)w^7(10\xi - w )$ and $a_{8}=w - \xi$. 
As $a_1>0$ and $a_{8}<0$, one has $w>10\xi $ and $\xi >w>10\xi$ 
which is a contradiction. If $w=5u$, then $a_6=20u^2(u+\xi )>0$ which is 
again a contradiction.

%
\end{proof}

\section{Proof of Lemma~\protect\ref{lmH4bis}
\protect\label{prlmbis}}

The multiplicities of the negative roots of $P$ define the following a priori 
possible cases: 

$${\rm A)}~~~(5,1)~,~~~{\rm B)}~~~(4,2)~~~
{\rm and~C)}~~~(3,3)~.$$ 
In all 
of them the proof is carried out simultaneously for the two possibilities 
$\{ a_1=a_4=0\}$ and $\{ a_1=a_5=0\}$. In order to simplify the proof we fix 
one of the roots to be equal to $-1$ (this can be achieved by a change 
$x\mapsto \beta x$, $\beta >0$, followed by $P\mapsto \beta ^{-9}P$). 
This allows to deal with one less  
parameter. By doing so we can no longer require that $a_{8}=-1$, 
but only that $a_{8}<0$.

\begin{proof}[Case A)]
We use the following parametrization: 

$$P=(x+1)^5(sx+1)(tx-1)^2(wx-1)~,~s>0~,~t>0~,~w>0~,~ t\neq w~,$$ 
i.e. the negative roots of $P$ are at $-1$ and $-1/s$ and the positive ones at 
$1/t$ and~$1/w$. 

The condition $a_1=w+2t-s-5=0$ yields $s=w+2t-5$. For $s=w+2t-5$, one has

$$\begin{array}{lcl}
a_3=a_{32}w^2+a_{31}w+a_{30}~,&&a_4=a_{42}w^2+a_{41}w+a_{40}~,\\ \\ 
{\rm where}&&a_{32}=-2t+5~,\\ 
a_{31}=-(2t-5)^2~,&&a_{30}=-2t^3+20t^2-50t+40\\ \\ 
{\rm and}&&a_{42}=t^2-10t+10~,\\ 
a_{41}=2t^3-25t^2+70t-50~,&&a_{40}=
-10t^3+55t^2-100t+45~.\end{array}$$
The coefficient $a_{30}$ has a single real root $6.7245\ldots$ hence $a_{30}<0$ 
for $t>6.7245\ldots$. On the other hand, 

$$a_{32}w^2+a_{31}w=w(-2t+5)(w+2t-5)=w(-2t+5)s$$ 
which is negative for $t>6.7245\ldots$. Thus the inequality $a_3>0$ fails for 
$t>6.7245\ldots$. Observing that $a_{41}=(2t-5)a_{42}$ one can write 

$$a_4=(w+2t-5)wa_{42}+a_{40}=swa_{42}+a_{40}~.$$

The real roots of $a_{42}$ (resp. $a_{40}$) equal 
$1.127\ldots$ and $8.872\ldots$ (resp. $0.662\ldots$). Hence 
for $t\in [1.127\ldots ,8.872\ldots ]$, the inequality $a_4>0$ 
fails. There remains to consider the possibility 
$t\in (0,1.127\ldots )$.

It is to be checked directly that for $s=w+2t-5$, one has

$$a_{8}/t=10t^2w+5tw^2-2t^2-29tw-2w^2+5t+10w=(5t-2)ws+t(5-2t)$$
which is nonnegative (hence $a_{8}<0$ fails) for $t\in [2/5,5/2]$. 
Similarly

$$\begin{array}{ll}
a_6=a_6^*w(w+2t-5)+a_6^{\dagger}=a_6^*ws+a_6^{\dagger}~~,~~&{\rm where}\\ 
a_6^*=10t^2-20t+5~~,~~&a_6^{\dagger}=
-5(t-1)(4t^2-9t+1)~.\end{array}$$
The real roots of $a_6^*$ (resp. $a_6^{\dagger}$) equal 
$1.707\ldots >2/5=0.4$ and $0.293\ldots$ (resp. 
$1 >2/5$, $0.117\ldots$ and 
$2.133\ldots$) hence for $t\in (0,2/5)$ one has 
$a_6^*>0$ and $a_6^{\dagger}>0$, i.e. $a_6>0$ and the equality $a_6=0$  
or the inequality $a_6<0$ is impossible.
\end{proof}

\begin{proof}[Case B)]
We parametrize $P$ as follows: 

$$P=(x+1)^4(Tx^2+Sx-1)^2(wx-1)~,~T>0~,~w>0~.$$
In this case we presume $S$ to be real, not necessarily positive. The factor 
$(Tx^2+Sx-1)^2$ contains the double positive and negative roots of $P$. 

From $a_1=w+2S-4=0$ one finds $S=(4-w)/2$. For $S=(4-w)/2$, one has 

$$\begin{array}{ll}
a_{8}/T=(4w-1)T+4w-w^2~,&\\ 
a_5=a_{52}T^2+a_{51}T+a_{50}~,&{\rm where}\\ 
a_{52}=w-4~,&\\ 
a_{51}=-4w^2+10w-16&{\rm and}\\ 
a_{50}=(3/2)w^3-9w^2+16w-12~.&\end{array}$$
Suppose first that $w>1/4$. The inequality $a_{8}<0$ is equivalent to 

$$T<T_0:=(w^2-4w)/(4w-1)~.$$
As $T>0$, this implies $w>4$.

For $T=T_0$, one obtains $a_5=3C/2(4w-1)^2$, where the numerator 
$C:=6w^5-40w^4+85w^3-54w^2+32w-8$ has a single real root $0.368\ldots$. 
Hence for $w>4$, one has $C>0$ and $a_5|_{T=T_0}>0$. On the other hand, 
$a_{50}=a_5|_{T=0}$ has a single real root $3.703\ldots$, so 
for $w>4$ one has $a_5|_{T=0}>0$. For $w>4$ fixed, and for 
$T\in [0,T_0]$, the value of the derivative 

$$\partial a_5/\partial T=(2w-8)T-4w^2+10w-16$$
is maximal for $T=T_0$; this value equals 

$$-2(7w^3-14w^2+21w-8)/(4w-1)$$
which is negative because the only real root of the numerator is $0.510\ldots$. 
Thus $\partial a_5/\partial T<0$ and $a_5$ is minimal for $T=T_0$. 
Hence the inequality $a_5<0$ fails for $w>1/4$. For $w=1/4$ one has 
$a_{8}=15/16>0$. 

So suppose that $w\in (0,1/4)$. In this case the condition $a_{8}<0$ implies 
$T>T_0$. For $T=T_0$ one gets 

$$a_4=3D/2(4w-1)^2~~~,~~~{\rm where}~~~D:=8w^5-32w^4+54w^3-85w^2+40w-6$$
has a single real root $2.719\ldots$. Hence for $w\in (0,1/4)$ one has 
$D<0$ and $a_4|_{T=T_0}<0$. The derivative 
$\partial a_4/\partial T=-w^2-2T-4$ being negative one has $a_4<0$ for 
$w\in (0,1/4)$, i.e. the inequality $a_4>0$ fails.
\end{proof}

\begin{proof}[Case C)]
We set 

$$P:=(x+1)^3(sx+1)^3(tx-1)^2(wx-1)~,~s>0~,~t>0~,~w>0~,~ t\neq w~.$$
The condition $a_1=w+2t-3s-3=0$ implies $s=s_0:=(w+2t-3)/3$. For $s=s_0$, one 
has 
$27a_{8}=t(w+2t-3)^2H^*$, where 
\begin{equation}\label{H}
H^*:=6wt^2-2t^2+3w^2t-5wt+3t+6w-2w^2~.
\end{equation}
We show first that for $s=s_0$, the case $a_1=a_5=0$ is impossible. To fix the ideas, we represent on Fig.~\ref{HA5} the sets $\{H^*=0\}$ (solid curve) and $\{a_5^*=0\}$ (dashed curve), where $a_5^*:=a_5|_{s=s_0}$. Although we need only the nonnegative values of $t$ and $w$, we show these curves also for the negative values of the variables to make things more clear. (The lines $t=2/3$ and $w=1/3$ are asymptotic lines for the set $\{H^*=0\}$). For $t \geq 0$ and $w \geq 0$, the only point, where $H^*=a_5^*=0$, is the point $(0;3)$. However, at this point one has $a_8=0$, i.e. this does not correspond to the required sign pattern.

\begin{figure}[htbp]
\vskip0.5cm
\centerline{\hbox{\includegraphics[scale=0.4]{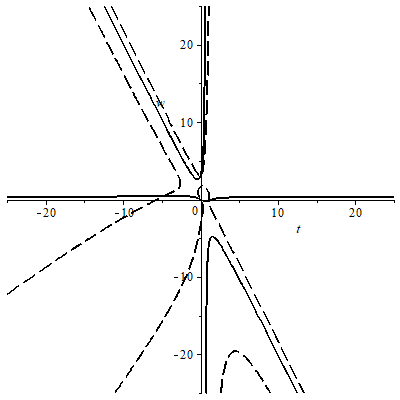}}\hskip1cm \hbox{\includegraphics[scale=0.4]{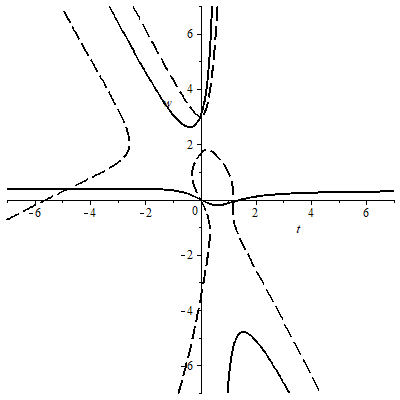}}}
    \caption{The sets $\{H^*=0\}$ (solid curve) and $\{a_5^*=0\}$ (dashed curve), with 3 and 4 connected components respectively. }
\label{HA5}
\end{figure}



\begin{lm}\label{lmH} 
(1) For $(t,w)\in \Omega_1 \cup \Omega_2$, where $\Omega_1 = [3/2,\infty )\times [1/3,\infty )$ and $\Omega_2= [0,3/2]\times [0,3]$, one has $H^* \geq0$. 

(2) For $(t,w)\in \Omega_3:=[3/2,\infty )\times [0,1/3]$, one has $a_5^*<0$.

(3) For $(t,w)\in \Omega_4:=[0,3/2]\times [3, \infty)$, the two conditions $H^*<0$ and $a_5^*=0$ do not hold simultaneously. 

\end{lm}

Lemma~ \ref{lmH} (which is proved after the proof of Lemma~ \ref{lmH4bis}) implies that in each of the sets $\Omega_j$, $1 \leq j \leq 4$, at least one of the two conditions $H^*<0$ (i.~e. $a_8<0$) and $a_5^*=0$ fails. There remains to notice that $\Omega_1 \cup \Omega_2 \cup  \Omega_3 \cup \Omega_4 = \{ t \geq 0, w \geq 0  \}$. 

Now, we show that for $s=s_0$, the case $a_1=a_4=0$ is impossible. On Fig.~\ref{HA4} we show the sets $\{H^*=0\}$ (solid curve) and $\{a_4^*=0\}$ (dashed curve), where $a_4^*:=a_4|_{s=s_0}$. We use the notation introduced in Lemma~\ref{lmH}. By part (1) of Lemma~\ref{lmH} the case $a_1=a_4=0$ is impossible for $(t,w)\in \Omega_1 \cup \Omega_2$. 

\begin{figure}[htbp]
\vskip0.5cm
\centerline{\hbox{\includegraphics[scale=0.4]{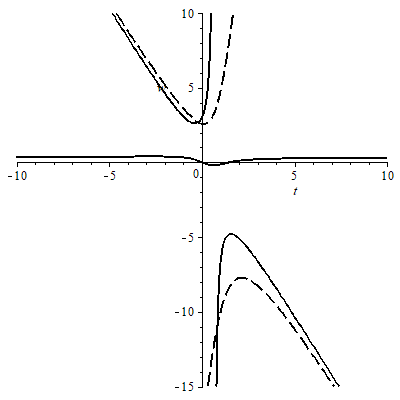}}\hskip1cm \hbox{\includegraphics[scale=0.4]{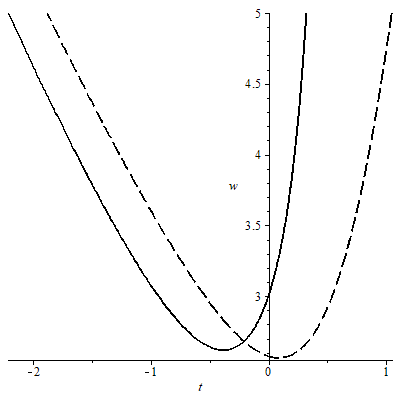}}}
    \caption{The sets $\{H^*=0\}$ (solid curve) and $\{a_4^*=0\}$ (dashed curve), with 3 and 2 connected components respectively. }
\label{HA4}
\end{figure}

\begin{lm}\label{lmHA4} 
(1) For $(t,w)\in \Omega_3$, one has $a_4^*>0$. 

(2) For $(t,w)\in \Omega_4$, the two conditions $H^*<0$ and $a_4^*=0$ do not hold simultaneously. 

\end{lm}   
Thus the couple of conditions $H^*<0$, $a_4^*=0$ fails for $t \geq 0$, $w \geq 0$. This proves Lemma~\ref{lmH4bis}. Lemma~\ref{lmHA4} is proved after Lemma~\ref{lmH} .    

\end{proof}

\begin{proof}[Proof of Lemma~\ref{lmH}]
Part (1). Consider the quantity $H^*$ as a polynomial in the variable $w$: 
$$\begin{array}{lll}
H^*=b_2w^2+b_1w+b_0~,&\rm{where}&b_2=3t-2~,\\ \\
b_1=6t^2-5t+6&\rm{and}&b_0=-2t(t-3/2)~.
\end{array}$$
Its discriminant $\Delta_w:=b_1^2-4b_0b_2=9(2t^2-3t+2)(2t^2+t+2)$ is positive for any real $t$. This is why for $t \neq 2/3$, the polynomial $H^*$ has $2$ real roots; for $t=2/3$, it is a linear polynomial in $w$ and has a single real root $-5/24$. When $H^*$ is considered as a polynomial in the variable $t$, one sets   
\begin{equation}\label{EqHt}
\begin{array}{lll}
H^*:=c_2t^2+c_1t+c_0~,&\rm{where}&c_2=6w-2~,\\ \\
c_1=3w^2-5w+3&\rm{and}&c_0=-2w(w-3)~.
\end{array}
\end{equation}
Its discriminant 

$$\Delta_t:=c_1^2-4c_0c_2=9(w^2+5w+1)(w^2-3w+1)$$
 is negative if and only if $w\in (-4.79\ldots ,0.20\ldots ) \cup (-0.38\ldots ,2,61\ldots )$. One checks directly that $H^*|_{w=1/3}=(5/3)t+16/9$ which is positive for $t \geq 0$. Next, one has $H^*|_{w=0} =b_0$ which is negative for $t > 3/2$. Finally, for $t > 3/2$, the ratio $b_0/b_2$ is negative which means that for $t>3/2$ fixed, the polynomial $H^*$ has one positive and one negative root, so the positive root belongs to the interval $(0,1/3)$ (because $H^*|_{w=1/3}>0$). Hence $H^* \geq 0$ for $(t,w)\in \Omega_1$ and $H^*>0$ for $(t,w)$ in the interior of  $ \Omega_1$.

Suppose now that $(t,w)\in [0,3/2]\times [0,3]$. For $t \in (2/3, 3/2]$ fixed, one has $b_2>0$, $b_1/b_2 >0$ and $b_0/b_2 > 0$ which implies that $H^*$ has two negative roots, and for $(t,w)\in (2/3,3/2]\times [0,3]$, one has $H^*>0$. For $t \in [0, 2/3)$ fixed, one has $b_2<0$, $b_1/b_2 <0$, $b_0/b_2 < 0$ and $H^*$ has a positive and a negative root; given that $b_2<0$, $H^*$ is positive between them. For $w=3$ and $t \geq 0$, one has $H^*=t(16t+15) \geq 0$, with equality only for $t=0$. Therefore $H^*>0$ for $(t,w)\in [0,2/3)\times [0,3]$. And for $t=2/3$, one obtains $H^*=(16/3)w+10/9$ which is positive for $w \geq 0$. 

Part (2). One has 
$$\begin{array}{ccl}
a_5^*&=&-8t^5+8t^4w+6t^3w^2-4t^2w^3-2tw^4-24t^4 \\ \\ 
&&-66t^3w-63t^2w^2-12tw^3+3w^4+84t^3+153t^2w \\ \\ 
&&+90tw^2-3w^3-144t^2-144tw-36w^2+108t+54w ~. 
\end{array}$$
Consider $a_5^*$ as a polynomial in $w$. Set $R_w:=$Res$(a_5^*,\partial a_5^* / \partial w, w) / 2125764$. Then $R_w=(2t-3)R_w^1R_w^2$, where 
$$\begin{array}{ccl}
R_w^1&=&32t^5+16t^4-80t^3+184t^2-142t-63~,\\ \\ 
R_w^2&=&10t^{10}-80t^9+365t^8-928t^7+1564t^6-1788t^5 \\ \\ 
&&+1345t^4-668t^3+208t^2-40t+4 ~. 
\end{array}$$
The real roots of $R_w^1$ (resp. $R_w^2$) equal $-2.56 \ldots$, $-0.30 \ldots$ and $1.18 \ldots$ (resp. $0.34 \ldots$ and $1.16\ldots$). That is, the largest real root of $R_w$ is $3/2$. One has 

$$a_5^*|_{w=0}=-4t(2t^4+6t^3-21t^2+36t-27)~,$$
with real roots equal to $-5.55 \ldots$, $0$ and $1.18 \ldots$. This means that for $t > 3/2$, the signs of the real roots of $a_5^*$ do not change and their number (counted with multiplicity) remains the same. For $t=3/2$ and $t=2$, one has 

$$a_5^*=-30w^3-(45/2)w^2-(243/4) \hspace{3mm}{\rm and}\hspace{3mm} a_5^*=-w^4-43w^3-60w^2-22w-328$$ respectively, which quantities are negative. Hence $a_5^*<0$ for $t \geq 3/2$ from which part (2) follows. 

Part (3). Consider the resultant 
$$\begin{array}{ccll}
R^{\flat}&:=&{\rm Res}(H^*,a_5^*,t)=-52488w(w-3)R^{\sharp}(w^2-w+1)^2~,&{\rm where}\\ \\ 
R^{\sharp}&:=&5w^6-16w^5+40w^4-23w^3+61w^2-16w-2~. 
\end{array}$$
The real roots of $R^{\sharp}$ equal $-0.09 \ldots$ and $0.37 \ldots$; the factor $w^2-w+1$ has no real roots. Thus the largest real root of $R^{\flat}$ equals $3$. For $w=3$, one has 

$$a_5^*=-4t^2(2t^3+15t+90) \leq 0~,$$
with equality if and only if $t=0$. For $w>3$ and $t \geq 0$, the sets \{$H^*=0$\} and \{$a_5^*=0$\} do not intersect (because $R^{\flat}<0$). We showed in the proof of part (1) of the lemma that the discriminant $\Delta_t$ is positive for  $w \ge 3$. Hence each horizontal line $w=w_0>3$ intersects the set  \{$H^*=0$\} for two values of $t$; one of them is positive and one of them is negative (because $c_0/c_2 < 0$); we denote them by $t_+$ and $t_-$. 

The discriminant $R_t:=$Res$(a_5^*,\partial a_5^* / \partial t, t)$ equals $2176782336(w-3)R_t^1R_t^2$, where 
 $$\begin{array}{ccl}
R_t^1&:=&5w^{12}+50w^{11}+100w^{10}-2513w^9+10781w^8-25932w^7+46604w^6\\ \\ 
&&-70411w^5+86678w^4-82706w^3+65264w^2-43104w+16896~, \\ \\
R_t^2&:=&8w^4+154w^3-68w^2-239w-352~. 
\end{array}$$
The factor $R_t^1$ is without real roots. The real roots of $R_t^2$ (both simple) equal $-19.61 \ldots$ and $1.81\ldots$. Hence for each $w=w_0>3$, the polynomial $a_5^*$ has one and the same number of real roots. Their signs do not change with $t$. Indeed, $a_5^*$ is a degree $5$ polynomial in $t$, with leading coefficient and constant term equal to $-8$ and $3w(w-3)(w^2+2w-6)$ respectively; the real roots of the quadratic factor equal $-3.64 \ldots$ and $1.64 \ldots$. 

For $w_0>3$, the polynomial $a_5^*$ has exactly 3 real roots $t_1<t_2<t_3$. For any $w_0>3$, the signs of these roots and of the roots $t_{\pm}$ of $H^*$ and the order of these 5 numbers on the real line are the same. For $w=4$, one has 

$$ t_1=-3.3 \ldots < t_-=-1.6 \ldots < t_2 = -0.8 \ldots < t_+=0.2 \ldots < t_3=0.3 \ldots $$          
Hence the only positive root $t_3$ of $a_5^*$ belongs to the domain where $H^*>0$. Hence one cannot have $a_5^*=0$ and $H^*<0$ at the same time. The lemma is proved. 

\end{proof}

\begin{proof}[Proof of Lemma~\ref{lmHA4}]
Part (1). One has 
 $$\begin{array}{ccl}
a_4^*&:=&-20t^4-22t^3w-30t^2w^2-10tw^3+w^4+66t^3+45t^2w+36tw^2\\ \\ 
&&+15w^3-135t^2-54tw-54w^2+108t+54w-81~. 
\end{array}$$
Consider $a_4^*$ as a polynomial in $t$. Its discriminant $\Delta_t^{\bullet}:=$Res$(a_4^*,\partial a_4^* / \partial t, t)$ is of the form $170061120 \Delta^{\flat} \Delta^{\sharp} (w^2-w+1)^2 $, where 
 $$\begin{array}{ccll}
\Delta^{\flat}&:=&9w^4+48w^3+82w^2+56w+205&{\rm and}\\ \\ 
\Delta^{\sharp} &:=&3w^4+14w^3-63w^2+51w-82~. 
\end{array}$$
Only the factor $\Delta^{\sharp}$ has real roots, and these equal $w_-:=-7.72 \ldots$ and $w_+:=2.56 \ldots$; they are simple. For $w \in (w_-,w_+)$, the quantity $a_4^*$ is negative. Indeed, $a_4^*|_{w=0}=-20t^4+66t^3-135t^2+108t-81$ which polynomial has no real roots; hence this is the case of  $a_4^*|_{w=w_0}$ for any $w_0 \in (w_-,w_+)$. This proves part (1), because the set $\Omega_3$ belongs to the strip $\{w_-<w<w_+\}$.   

Part (2). The discriminant Res$(a_4^*,H^*,t)$ equals $-26244R^{\triangle}(w^2-w+1)^2$ whose factor 

$$R^{\triangle}:=2w^6+16w^5-61w^4+23w^3-40w^2+16w-5$$
has exactly two real (and simple) roots which equal $-10.90 \ldots$ and $2.68 \ldots$. Hence for $w \geq 3>w_+$, 
\vspace{1mm}

(1) the sets $\{ H^*=0 \}$ and $\{ a_4^*=0 \}$ do not intersect;
\vspace{1mm}

(2) the numbers of positive and negative roots of $H^*$ and $a_4^*$ do not change; for $H^*$ this follows from formula (\ref{EqHt}); for $a_4^*$ whose leading coefficient as a polynomial in $t$ equals $-20$, this results from $a_4^*|_{t=0}=w^4+15w^3-54w^2+54w-81$ whose real roots $-18.1\ldots$ and $2.5 \ldots$ (both simple) are $<3$.         
\vspace{1mm}

Hence for $w =w_0 \geq 3$, one has $h_-<A_-<0 \leq h_+ <A_+$, where $h_-$ and $h_+$ (resp. $A_-$ and $A_+$) are the two roots of $H^*|_{w=w_0}$ (resp. of $a_4^*|_{w=w_0}$), with equality only for $w_0=3$. It is sufficient to check this string of inequalities for one value of $w_0$, say, for $w_0=4$, in which case one obtains 

$$h_-=-1.63 \ldots <A_-=-1.26 \ldots <  h_+=0.22 \ldots <A_+=0.85 \ldots~.$$
Hence for $w=w_0 \geq 3$, the only positive root of the polynomial $a_4^*|_{w=w_0}$ belongs to the domain $\{ H^*>0 \}$. This proves part (2) of the Lemma.

\end{proof}

\section{Proofs of Lemmas~\protect\ref{lmH4ter}, 
\protect\ref{lmH5} and \protect\ref{lmH6}
\protect\label{prlmter}}

\begin{proof}[Proof of Lemma~\ref{lmH4ter}:]
\begin{nota}\label{notaP}
{\rm If $\zeta _1$, $\zeta _2$, $\ldots$, $\zeta _k$ are distinct roots 
of the polynomial $P$ (not necessarily simple), then by $P_{\zeta _1}$, 
$P_{\zeta _1,\zeta _2}$, $\ldots$, $P_{\zeta _1,\zeta _2,\ldots ,\zeta _k}$ we denote the 
polynomials} 

$$P/(x-\zeta _1)~,~P/(x-\zeta _1)(x-\zeta _2)~,~\ldots~,~P/(x-\zeta _1)(x-\zeta _2)\ldots (x-\zeta _k)~.$$
\end{nota}

Denote by $u$, $v$, $w$ and $t$ the four distinct roots of $P$ (all nonzero). 
Hence 

$$P=(x-u)^m(x-v)^n(x-w)^p(x-t)^q~,~m+n+p+q=9~.$$ 
For $j=1$, $4$ or $5$, we show that the Jacobian $3\times 4$-matrix 
$J:=(\partial (a_{8},a_7,a_j)/\partial (u,v,w,t))^t$ 
(where $a_{8}$, $a_7$, $a_j$ are the corresponding coefficients of $P$ 
expressed as functions of $(u,v,w,t)$) is of rank $3$. 
(The entry in position $(2,3)$ of $J$ is $\partial a_7/\partial w$.) Hence 
one can vary the values of $(u,v,w,t)$ in such a way that $a_{8}$ and $a_j$ 
remain fixed (the value of $a_{8}$ being $-1$) and $a_7$ takes all possible 
nearby values. Hence the polynomial is not $a_7$-maximal. 

The entries of the four columns of $J$ are the coefficients of $x^{8}$, $x^7$ 
and $x^j$ of the polynomials 
$-mP_u=\partial P/\partial u$, $-nP_v$, $-pP_w$ and $-qP_t$. 
By abuse of language 
we say that the linear space $\mathcal{F}$ spanned 
by the columns of $J$ is generated by the polynomials 
$P_u$, $P_v$, $P_w$ and $P_t$. As 

$$P_{u,v}=\frac{P_u-P_v}{v-u}~~~\, ,~~~\, P_{u,w}=\frac{P_u-P_w}{w-u}~~~\,  {\rm and}~~~\, P_{u,t}=\frac{P_u-P_t}{t-}~~~\, ,~~~\, $$ 
one can choose as generators of $\mathcal{F}$ 
the quadruple ($P_u$, $P_{u,v}$, $P_{u,w}$, $P_{u,t}$); in the same 
way one can choose ($P_u$, $P_{u,v}$, $P_{u,v,w}$, $P_{u,v,t}$) or 
($P_u$, $P_{u,v}$, $P_{u,v,w}$, $P_{u,v,w,t}$) (the latter polynomials 
are of respective 
degrees $8$, $7$, $6$ and $5$). As $(x-t)P_{u,v,w,t}=P_{u,v,w}$, 
$(x-w)P_{u,v}=P_{u,v,w}$ etc. one can choose as generators the quadruple 

$$\psi :=(x^3P_{u,v,w,t}~,~x^2P_{u,v,w,t}~,~xP_{u,v,w,t}~,~P_{u,v,w,t})~.$$ 
Set $P_{u,v,w,t}:=x^5+Ax^4+\cdots +G$. The coefficients 
of $x^{8}$, $x^7$ and $x^5$ of the quadruple $\psi$ define the matrix 
$J^*:=\left( \begin{array}{cccc}1&0&0&0 \\A&1&0&0 \\ D&C&B&A\end{array}
\right)$. 
Its columns span the space $\mathcal{F}$ hence rank$\, J^*=$rank$\, J$. 
As at least one of 
the coefficients $B$ and $A$ is nonzero (Lemma~\ref{lmhyp}) one has 
rank$\, J^*=3$ and the lemma follows (for the case $j=6$). In the 
cases $j=5$ and 
$j=1$ the last row of $J^*$ equals respectively $(~E~D~C~B~)$ and 
$(~0~0~G~F~)$ and in the same way rank$\, J^*=3$.

\end{proof}

\begin{proof}[Proof of Lemma~\ref{lmH5}:]
We are using Notation~\ref{notaP} and the method of proof of 
Lemma~\ref{lmH4ter}. Denote by $u$, $v$, $w$, $t$, $h$ the 
five distinct real roots of $P$ (not necessarily simple). 
Thus using Lemma~\ref{lmHtriple} one can assume that

\begin{equation}\label{possible}
P=(x+u)^{\ell}(x+v)^m(x+w)^n(x-t)^2(x-h)~~~,~~~u, v, w, t, h > 0~~~,~~~
\ell +m+n=6~.\end{equation}
Set 
$J:=(\partial (a_{8},a_7,a_j,a_1)/\partial (u,v,w,t,h))^t$, $j=4$ or $5$. 
The columns of $J$ span a linear space $\mathcal{L}$ defined by 
analogy with the space $\mathcal{F}$ of the proof of Lemma~\ref{lmH4ter}, 
but spanned by $4$-vector-columns.

Set $P_{u,v,w,t,h}:=x^4+ax^3+bx^2+cx+d$. Consider the vector-column 
$$(0,0,0,0,1,a,b,c,d)^t~.$$
The similar vector-columns defined when using the polynomials $x^sP_{u,v,w,t,h}$, 
$1\leq s\leq 4$, instead of $P_{u,v,w,t,h}$ 
are obtained from this one by successive shifts 
by one position upward. To obtain generators of $\mathcal{L}$ one has 
to restrict these vector-columns to the rows corresponding to $x^{8}$ 
(first), $x^7$ (second), $x^j$ ($(9-j)$th) and $x$ (eighth row). 

Further we assume that $a_1=0$. 
If this is not the case, then at most one of the 
conditions $a_4=0$ and $a_5=0$ is fulfilled and the proof of the lemma can be 
finished by analogy with the proof of Lemma~\ref{lmH4ter}. 

Consider the case $j=5$. 
Hence the rank of $J$ is the same as the 
rank of the matrix 

$$M:=\left( \begin{array}{ccccc}1&0&0&0&0\\ a&1&0&0&0\\ c&b&a&1&0\\ 
0&0&0&d&c\end{array}\right) ~~~
\begin{array}{c}x^{8}\\ x^7\\ x^5\\ x\end{array}~.$$
One has rank$\, M=2+$rank$\, N$, where 
$N=\left( \begin{array}{ccc}a&1&0\\ 0&d&c\end{array}\right)$. Given that 
$d\neq 0$, see Lemma~\ref{lmnot0bis}, one can have rank$\, N<2$ only if $a=c=0$. We show that the 
condition $a=c=0$ leads to the contradiction that one must have $a_{8}>0$. 
We set $u=1$ to reduce the number 
of parameters, so we require only the inequality $a_{8}<0$, 
but not the equality $a_{8}=-1$, to hold true. We have to consider the following cases for 
the values of the triple $(\ell ,m,n)$ 
(see (\ref{possible})): 
1) $(4,1,1)$, 2) $(3,2,1)$ and 3) $(2,2,2)$. Notice that 

$$P_{u,v,w,t,h}|_{u=1}=(x+1)^{\ell -1}(x+v)^{m-1}(x+w)^{n-1}(x-t)~.$$

In case 1) one has 
\begin{equation}\label{case1}
a=3-t~,~b=3-3t~,~c=1-3t~{\rm and}~d=-t~, 
\end{equation}
so the condition $a=c=0$ leads to the contradiction $3=t=1/3$. 

In case 2) one obtains 
\begin{equation}\label{case2}
a=2+v-t~,~b=1+2v-(2+v)t~,~c=v-(1+2v)t~{\rm and}~d=-vt~. 
\end{equation}
Thus, the condition $a=c=0$ yields $v=-1$, $t=1$. This is also a contradiction because $v$ must be positive. 

In case 3) one gets 
\begin{equation}\label{case3}
\begin{array}{lcl}
a=1+v+w-t~,&&b=v+(1+v)w-(1+v+w)t~, \\ \\
c=vw-(v+(1+v)w)t&~~~{\rm and}~~~&d=-vwt~.
\end{array} 
\end{equation}
When one expresses $v$ and $w$ as functions of $t$ from the system of equations $a=c=0$, one obtains two possible solutions: $v=t$, $w=-1$ and $v=-1$, $w=t$. In both cases one of the variables ($v$, $w$) is negative which is a contradiction.

Now consider the case $j=4$. The matrices $M$ and $N$ equal respectively 

$$M:=\left( \begin{array}{ccccc}1&0&0&0&0\\ a&1&0&0&0\\ d&c&b&a&1\\ 
0&0&0&d&c\end{array}\right) ~~~~,~~~
N=\left( \begin{array}{ccc}b&a&1\\ 0&d&c\end{array}\right) ~.$$

One has rank$\, N<2$ only for $b=0$, $d=ac$ (because $d \neq 0$). 

In case 1) these conditions lead to the contradiction $1=t=(3/2) \pm \sqrt{5}/2$ see (\ref{case1}). 

In case 2) one expresses the variable $t$ from the condition $b=0$: $t=t^{\bullet}:=(1+2v)/(2+v)$. Set $a^{\bullet}:=a|_{t=t^{\bullet}}$, $c^{\bullet}:=c|_{t=t^{\bullet}}$  and $d^{\bullet}:=d|_{t=t^{\bullet}}$. The quantity $d^{\bullet}-a^{\bullet}c^{\bullet}$ equals $3(v^2+v+1)^2/(2+v)^2$ which vanishes for no $v \geq 0$. So case 2) is also impossible.    

In case 3) the condition $b=0$ implies $t=t^{\triangle}:=(vw+v+w)/(1+v+w)$. Set $a^{\triangle}:=a|_{t=t^{\triangle}}$, $c^{\triangle}:=c|_{t=t^{\triangle}}$  and $d^{\triangle}:=d|_{t=t^{\triangle}}$. The quantity $d^{\triangle}-a^{\triangle}c^{\triangle}$ equals $(w^2+w+1)(v^2+v+1)(v^2+vw+w^2)/(1+v+w)^2$ which is positive for any $v \geq 0$, $w \geq 0$. Hence case 3) is impossible. The lemma is proved.      

\end{proof}

\begin{proof}[Proof of Lemma~\ref{lmH6}:]
We use the same ideas and notation as in the proof of Lemma~\ref{lmH5}. 
Six of the six or more real roots of $P$ are denoted by $(u,v,w,t,h,q)$. 
The space $\mathcal{L}$ is defined by analogy with the one of the 
proof of Lemma~\ref{lmH5}. 
The Jacobian matrix $J$ is of the form 

$$J:=(\partial (a_{8},a_7,a_j,a_1)/\partial (u,v,w,t,h,q))^t~.$$ 
Set $P_{u,v,w,t,h,q}:=x^3+ax^2+bx+c$ and consider  
the vector-column  

$$(0,0,0,0,0,1,a,b,c)^t~.$$
Its successive shifts by one position upward correspond to the polynomials 
$x^sP_{u,v,w,t,h,q}$, $s\leq 5$. In the case $j=5$ 
the matrices $M$ and $N$ 
look like this:

$$M=\left( \begin{array}{cccccc}1&0&0&0&0&0\\ a&1&0&0&0&0\\ 
c&b&a&1&0&0\\ 0&0&0&0&c&b\end{array}\right)~~~\, , ~~~\, 
N=\left( \begin{array}{cccc}a&1&0&0\\ 0&0&c&b\end{array}\right)~.$$
One has rank$\, M=2+$rank$\, N$ and rank$\, N=2$, 
because at least one of 
the two coefficients $b$ and $c$ is nonzero (Lemma~\ref{lmhyp}). Hence 
rank$\, M=4$ and the lemma is proved by analogy with Lemmas~\ref{lmH4ter} and 
\ref{lmH5}. 
In the case $j=4$ the matrices $M$ and $N$ look like this: 
$$M=\left( \begin{array}{cccccc}1&0&0&0&0&0\\ a&1&0&0&0&0\\ 
0&c&b&a&1&0\\ 0&0&0&0&c&b\end{array}\right)~~~\, , ~~~\, 
N=\left( \begin{array}{cccc}b&a&1&0\\ 0&0&c&b\end{array}\right)~.$$
The matrix $N$ is of rank $4$, because either $b \neq 0$ or $b=0$ and both $a$ and $c$ are nonzero (Lemma~\ref{lmhyp}). Hence rank $M=4$.   

\end{proof}

\section{Proof of part (2) of Theorem~\protect\ref{maintm}\protect\label{secpr2}}

We remind that we consider polynomials with positive leading coefficients. For $d=9$, we denote by $\sigma$ a sign pattern and by $\sigma^*$ the shortened sign pattern (obtained from $\sigma$ by deleting its last component)   

\begin{lm}\label{lm9s2} 
For $d=9$, if $pos \geq 2$ and $neg \geq 2$, then such a couple (sign pattern, admissible pair) is realizable. 
\end{lm}
 
\begin{proof}
Suppose that the last two components of $\sigma$ are equal (resp. different). Then the pair ($pos$, $neg-1$) (resp. ($pos-1$, $neg$))  is admissible for the sign pattern $\sigma^*$ and the couple $(\sigma^*,~(pos~, ~neg-1))$ (resp. $(\sigma^*,~(pos-1~, ~neg))$) is realizable by some degree $8$ polynomial $P$, see Remark~\ref{rem8g1}. Hence the couple $(\sigma, ~(pos ~,~neg))$ is realizable by the concatenation of the polynomials $P$ and $x+1$ (resp. $P$ and $x-1$).
 \end{proof}
 
 Lemma~\ref{lm9s2} implies that in any nonrealizable couple with $pos > 0$ and $neg > 0$, one of the numbers $pos$, $neg$ equals 1. Using the  the standard $\mathbb{Z}_2\times \mathbb{Z}_2$-action (i.e changing if necessary $P(x)$ to $-P(-x)$) one can assume that $pos=1$. This implies that the last component of the sign pattern is~$-$. 
 
\begin{lm}\label{lm9s3} 
For $d=9$, if $pos =1$, $neg \geq 2$ and the last two components of $\sigma$ are $(-~,~-)$, then such a couple $(\sigma, ~(pos ~,~neg))$ is realizable.
\end{lm}

\begin{proof}
The couple $(\sigma^*,~(pos~, ~neg-1))$ is realizable by some polynomial $P$, see Remark~\ref{rem8g1}. Hence the concatenation of $P$ and $x+1$ realizes the couple $(\sigma, ~(pos ~,~neg))$. 
\end{proof}

Hence for any nonrealizable couple $(\sigma, ~(pos ~,~neg))$, one has $pos=1$, $neg \geq 2$ and the last two components of $\sigma$ are $(+~,~-)$. Thus, the couple $(\sigma^*,~(0~, ~neg))$ is nonrealizable; The first and the last components of $\sigma^*$ are $+$. There are 19 such couples modulo the  $\mathbb{Z}_2\times \mathbb{Z}_2$-action, see~\cite{Ko2}:

$$\begin{array}{lllll}
{\rm Case}&&{\rm Sign~pattern}&&{\rm Admissible~pair(s)} \\ \\
A&&(+ + - - - - - + +)&&(0,6) \\ \\
B&&(+ - - - - - - + +)&&(0,6) \\ \\
C&&(+ + + + - - - - +)&&(0,6) \\ \\
D&&(+ + + - - - - - +)&&(0,6) \\ \\
E&&(+ - + - - - + - +)&&(0,2) \\ \\
F&&(+ - + - + - - - +)&&(0,2) \\ \\
G1-G2&&(+ - + - - - - - +)&&(0,2)~,~(0,4) \\ \\
H1-H2&&(+ - - - + - - - +)&&(0,2)~,~(0,4) \\ \\
I1-I3&&(+ - - - - - - - +)&&(0,2)~,~(0,4)~,~(0,6) \\ \\
\end{array}$$

$$\begin{array}{lllll}
J&&(+ + + - - - - + +)&&(0,6) \\ \\
K&&(+ - - - - + - - +)&&(0,4) \\ \\
L&&(+ - - - - - - + +)&&(0,4) \\ \\
M&&(+ - + + - - - - +)&&(0,4) \\ \\
N&&(+ - + - - - - + +)&&(0,4) \\ \\
Q&&(+ - - - - + - + +)&&(0,4) \\ \\
\end{array}$$

\noindent To obtain all couples $(\sigma^*,~(0~, ~neg))$ giving rise to nonrealizable couples $(\sigma,~(1~, ~neg))$ by concatenation with $x-1$, one has to add to the above list of cases ($A-Q$) the cases obtained from them by acting with the first generator of the $\mathbb{Z}_2\times \mathbb{Z}_2$-action, i.e. the one replacing $\sigma$ by $\sigma^r$, see Definition~\ref{z2act}. The second generator (the one replacing $\sigma$ by $\sigma^m$) has to be ignored, because it exchanges the two components of the admissible pair and the condition $pos=1$ could not be maintained. The cases that are to be added are denoted by ($A^r-Q^r$). E.g. 

$$N^r~~~\,~~~\,(+ + - - - - + - +)~~~\,~~~\,(0,4)~.$$
One can observe that, due to the center-symmetry of certain sign patterns, one has $A=A^r$, $E=E^r$, $Hj=Hj^r$, $j=1,~2$ and $Ij=Ij^r$, $j=1,~2,~3$. 

With the only exception of case $C^r$, we show that all cases ($A-Q$) and ($A^r-Q^r$), are realizable which proves part (2) of the theorem. We do this by means of Lemma~\ref{lmconcat}. We explain this first for the following cases:
$$\begin{array}{llllllllllll}
B^r~,&C~,&D~,&E~,&F~,&F^r~,&G1~,&G1^r~,&G2~,&G2^r~,&H1~,&H2~,\\ \\
I1~,&I2~,&I3~,& K~,&K^r~,&L^r~,&M~,&M^r~,&N^r&{\rm and}&Q^r~.
\end{array}$$
In all of them the last three components of $\sigma$ are $(- + -)$, and we set $P_2^{\dagger}:=x^2-x+1$ (see part (2) of Lemma~\ref{lmconcat}). The polynomial $P_2^{\dagger}$ has no real roots and defines the sign pattern $\sigma^{\dagger}:=(+ - +)$. Denote by $\tilde{\sigma}$ the sign pattern obtained from $\sigma$ by deleting its two last components. Hence $(1,neg)$ is an admissible pair for the sign pattern $\tilde{\sigma}$, and the couple $(\tilde{\sigma},(1,neg))$ is realizable by some degree $7$ monic polynomial $\tilde{P_1}$, see Remark~\ref{rem8g1}. By Lemma~\ref{lmconcat} the concatenation of $\tilde{P_1}$ and $P_2^{\dagger}$ realizes the couple $(\sigma,(1,neg))$. 

In cases $A$, $B$, $J$, $L$, $N$ and $Q$, the last four components of the sign pattern $\sigma$ are $(- + + -)$. We set $P_2^\triangle:=(x+2)((x^2-2)+1)=x^3-2x^2-3x+10$. Hence $P_2^\triangle$ realizes the couple $((+--+),(0,1))$. Denote by $\sigma^\triangle$ the sign pattern obtained from $\sigma$ by deleting its three last components. Hence $(1,neg-1)$ is an admissible pair for the sign pattern $\sigma^\triangle$, and the couple $(\sigma^\triangle,(1,neg-1))$ is realizable by some degree $6$ monic polynomial $P_1^\triangle$, see Remark~\ref{rem8g1}. By Lemma~\ref{lmconcat} the concatenation of $P_1^\triangle$ and $P_2^\triangle$ realizes the couple $(\sigma,(1,neg))$.

In the two remaining cases $D^r$ and $J^r$, the last six components of $\sigma$ are $(- - + + + -)$. The sign pattern $\sigma^{\ddagger}:=(+ + - - - +)$ is realizable by some degree $5$ polynomial $P_2^\ddagger$, see \cite{AlFu}. Denote by $\sigma^{\diamond}$ the sign pattern obtained from $\sigma$ by deleting its five last components. Hence in cases $D^r$ and $J^r$ one has $\sigma^{\diamond}=(+ - - - -)$ and $\sigma^{\diamond}=(+ + - - -)$ respectively. Thus the couple $(\sigma^{\diamond},(1,3))$ is realizable by some monic degree $4$ polynomial $P_1^\diamond$ (see Remark~\ref{rem8g1}), and the concatenation of $P_1^\diamond$ and $P_2^\ddagger$ realizes the couple $(\sigma,(1,neg))$. Part (2) of Theorem~\ref{maintm} is proved.

%
%

\end{document}